\documentclass[]{amsart}
\usepackage{amsmath}
\usepackage[english]{babel}
\usepackage[utf8]{inputenc}
\usepackage{amsfonts}
\usepackage{amsthm}
\usepackage{color}
\numberwithin{equation}{section}

\newtheorem{thm}{Theorem}[section]

\newtheorem{prop}[thm]{Proposition}

\newtheorem{lem}[thm]{Lemma}

\theoremstyle{remark}
\newtheorem{rmk}[thm]{Remark}
\theoremstyle{definition}
\newtheorem{defn}{Definition}[section]

\DeclareMathOperator{\Cov}{Cov}
\DeclareMathOperator{\E}{\mathbb{E}}
\DeclareMathOperator{\C}{\mathbb{C}}

\DeclareMathOperator{\N}{\mathbb{N}}
\DeclareMathOperator{\R}{\mathbb{R}}
\DeclareMathOperator{\bH}{\mathbb{H}}
\DeclareMathOperator{\cF}{\mathcal{F}}

\DeclareMathOperator{\cN}{\mathcal{N}}

\DeclareMathOperator{\cL}{\mathcal{L}}
\DeclareMathOperator{\cH}{\mathcal{H}}

\DeclareMathOperator{\cS}{\mathcal{S}}

\DeclareMathOperator{\cD}{\mathcal{D}}

\DeclareMathOperator{\bP}{\mathbb{P}}

\newcommand{\pd}[2]{\frac{\partial #1}{\partial #2}}
\newcommand{\pdsup}[3]{\frac{\partial^{#3} #1}{\partial #2^{#3}}}

\newcommand{\dersup}[3]{\frac{d^{#3} #1}{d #2^{#3}}}

\newcommand{\Norm}[2]{\left\Vert #1 \right\Vert_{#2}}

\author{Giacomo Ascione$^\ast$}
\address{$^\ast$ Dipartimento di Matematica e Applicazioni ``Renato Caccioppoli'', Universita degli Studi di Napoli Federico II, 80126 Napoli, Italy}
\author{Yuliya Mishura$^\odot $}
\address{$^\odot $ Department of Probability Theory, Statistics and Actuarial Mathematics, Taras Shevchenko National University of Kyiv, Volodymyrska 64, Kyiv 01601, Ukraine}
\author{Enrica Pirozzi$^\ast$}
\email{giacomo.ascione@unina.it \\
	myus@univ.kiev.ua \\
	enrica.pirozzi@unina.it}
\title[Convergence results for the TCfOU]{Convergence results for the Time-Changed fractional Ornstein-Uhlenbeck processes}
\begin{document}
\maketitle
\begin{abstract}
In this paper we study some convergence results concerning the one-dimensional distribution of a time-changed fractional Ornstein-Uhlenbeck process. In particular, we establish  that, despite the time change, the process admits a Gaussian limit random variable. On the other hand, we prove that the process converges towards the time-changed Ornstein-Uhlenbeck as the Hurst index $H \to 1/2^+$, with locally uniform convergence of one-dimensional distributions. Moreover, we also achieve convergence in the Skorohod $J_1$-topology  of the time-changed fractional Ornstein-Uhlenbeck process as $H \to 1/2^+$ in the space of c\'{a}dl\'{a}g functions. Finally, we exploit some convergence properties of mild solutions of a generalized Fokker-Planck equation associated to the aforementioned processes, as $H \to 1/2^+$ .
\end{abstract}
\section{Introduction}
Among stochastic processes, the Ornstein-Uhlenbeck (OU) process is one of the most commonly used in various applications. In particular, it is revealed to be a quite tractable process, being the solution of a simple linear stochastic differential equation 
\begin{equation}\label{SDEOU}
dU(t)=-\frac{1}{\theta}U(t)dt+\sigma dB(t),
\end{equation}
where $\theta,\sigma \in \R^+$ and $=\{B(t),t\ge 0\}$ is a standard Brownian motion. Such process exhibits Markov property and its covariance admits an exponential decay. For this reason, it cannot be used in models in which the memory, in particular, a long memory,  plays a prominent role. An example of long memory  behaviour is given by models of neurons of the prefrontal cortex, in which the spiking behaviour cannot be described by the OU process (see \cite{shinomoto1999ornstein}), that  was instead the main process   arising in stochastic Leaky Integrate-and-Fire models, which were the most used. In such context, a modification of the behaviour of the covariance of the process revealed to be effective to describe the spiking dynamics (see \cite{sakai1999temporally}). For this reason, the OU process had to be generalized in order to achieve some more long-range memory in the covariance.\\

One of the generalizations is given in terms of fractional Pearson diffusions (as considered in \cite{leonenko2013fractional}). However, due to ambiguous notations, we will refer to such process as $\alpha$-stable time-changed OU process. These time-changed processes were not Markov process, however they still preserve Markov property in some specific times, in particular, they exhibit a semi-Markov property. Another important thing we want to underline concerning $\alpha$-stable time-changed OU processes is that they always exhibit  long-range dependence, as shown in \cite{leonenko2013correlation}. Another generalization can be achieved by defining a time-changed OU process by means of a general inverse subordinator, as given in \cite{gajda2015time}. Let us stress out that the behaviour of such time-changed process is adapted   to the behaviour of some particular neurons which could not be modelled via OU processes (see \cite{ascione2019semi}).\\

Another interesting generalization can be obtained by changing the nature of the noise in equation \eqref{SDEOU}. In \cite{cheridito2003fractional}, the fractional OU (fOU) process has been introduced as the solution of the fractional Brownian motion (fBm)-driven equation
\begin{equation*}
dU_H(t)=-\frac{1}{\theta}U_H(t)dt+\sigma dB_H(t),
\end{equation*}
where $\theta,\sigma \in \R^+$ and $B_H\{B_H(t),t\ge 0\}$ is a fBm with Hurst parameter $H \in (0,1)$. This process exhibits long- or short-range dependence as a function of the Hurst parameter (see \cite{cheridito2003fractional,kaarakka2011fractional}). Considering  $H>1/2$ throughout the paper, we will always have long-range dependence. The approach by fOU processes is widely used in financial markets (see for instance \cite{anh2005financial}). In this direction,  the study of the fractional Cox-Ingersoll-Ross process has to be carried on, in particular referring to the hitting time at $0$ of the fOU process (see \cite{mishura2018stochastic,mishura2018fractional}). \\

One can actually combine the two strategies, obtaining a time-changed fractional Ornstein-Uhlenbeck process, as done in \cite{ascione2019time}. Such time-changed fOU process generally admits a probability density function, depending on both the Hurst index of the parent fOU process and the Bernstein function representing the Laplace exponent of the involved subordinator. Thus, as it is known that the fOU process converges towards the OU process as $H \to 1/2$ (actually, the covariance of the fOU is a continuous function with respect to the Hurst index, as shown in \cite{ascione2019fractional}), a natural question that arises is linked to the behaviour of the density of the time-changed fOU process with respect to the Hurst index.\\

In this paper we focus on some convergence results concerning the time-changed fOU process. On the one hand, since the fOU process is a Gaussian process, we   prove that the time-changed fOU process admits a limit distribution and that such distribution is Gaussian. This is justified by the fact that the time-change acts as a change of time scale and then does not modify the asymptotic behaviour. On the other hand, we   establish the convergence of the time-changed fOU introduced in \cite{ascione2019time} to the time-changed OU process introduced in \cite{gajda2015time}, as $H \to 1/2$. As one can  see, we have not only weak one-dimensional convergence of the process, but in fact the uniform convergence on compact sets of the marginal of the process, as $H \to 1/2$. Moreover, by using a continuous mapping approach, we are also able to show a functional limit theorem, obtaining the aforementioned convergence in the space of c\'{a}dl\'{a}g functions with respect to the Skorohod topology $J_1$. Finally, we study some convergence properties of the solutions of the generalized Fokker-Planck equation introduced in \cite{ascione2019time} and studied more in details in \cite{ascione2020alpha}.

The paper is structured as follows: in Section \ref{Sec1} we recall some properties of inverse subordinators and Bernstein functions. In Section \ref{Sec2} we recall the definition and some properties of the time-changed fOU process as given in \cite{ascione2019time} and    state some preliminary results concerning the convergence of the marginals of the fOU process.
	 In Section \ref{Sec3} we exploit the fact that the time-changed fOU process still admits a Gaussian limit distribution, despite not being a Gaussian process.
  In Section \ref{Sec4} we establish  the one-dimensional convergence of the time-changed fOU process towards the time-changed OU process as $H \to 1/2$. Moreover, we stress that the convergence of the density is actually uniform on compact sets and the convergence of the absolute moments is uniform.
  In Section \ref{Sec5} we provide a functional limit theorem for the convergence of the time-changed fOU process towards the time-changed OU process. To do this, we first prove a functional limit theorem for the convergence of the fOU process towards the OU process in the space of continuous functions (actually, in the space of H\"older-continuous functions), and then  give our main functional limit theorem, where we had to weaken the request on the topology to use the continuous mapping approach.
	 Finally, in Section \ref{Sec6} we show that limits of mild solutions of the generalized Fokker-Planck equation associated to the time-changed fOU as $H \to 1/2^+$ are actually solutions of the generalized Fokker-Planck equation associated to the time-changed OU process.
 
\section{Inverse subordinators and Bernstein functions}\label{Sec1}
In what  follows we will consider subordinators, i.e.increasing and starting from zero (hence positive) L\'evy processes (see \cite[Chapter $3$]{bertoin1996levy}). Each subordinator admits a Laplace exponent $\Phi(\lambda)$, i.e. a function $\Phi:[0,+\infty) \to \R$ such that
\begin{equation*}
\E[e^{-\lambda \sigma(t)}]=e^{-t\Phi(\lambda)}, \ t\ge 0, \ \lambda \ge 0.
\end{equation*}
In particular, such Laplace exponents $\Phi$ belong to the convex cone of Bernstein functions (see \cite{schilling2012bernstein}), and therefore we can define the characteristic triple of $\Phi$, given by $(a_\Phi,b_\Phi,\nu_\Phi)$, where $a_\Phi,b_\Phi\ge 0$ are constants and $\nu_\Phi$ is a L\'evy measure on $(0,+\infty)$ such that $\int_0^{+\infty}(t \wedge 1)\nu_\Phi(dt)<+\infty$. Indeed, let us recall that any Bernstein function $\Phi$ can be represented in a unique way by means of the characteristic triple as
\begin{equation*}
\Phi(\lambda)=a_\Phi+b_\Phi\lambda+\int_0^{+\infty}(1-e^{-t\lambda})\nu_\Phi(dt).
\end{equation*}
Here we will consider $\Phi$ to be such that $a_\Phi=0$ and if $b_\Phi=0$ then we assume that   $\nu_\Phi(0,+\infty)=+\infty$. Each Bernstein function $\Phi$ determines a unique (non-killed, since $a=0$) subordinator $\sigma_\Phi(t)$ whose Laplace exponent is $\Phi$. Now let us define the inverse subordinator associated to $\Phi$ as
\begin{equation*}
E_\Phi(t):=\inf\{y>0: \ \sigma_\Phi(y)>t\}.
\end{equation*}
It was established  in \cite{meerschaert2008triangular} that  our hypotheses on $\Phi$ are enough to guarantee that $E_\Phi(t)$ admits a probability density function $f_\Phi(s;t)$ for each $t>0$. Moreover, it has been shown that
\begin{equation*}
\cL_{t \to \lambda}[f_\Phi(s;t)]=\frac{\Phi(\lambda)}{\lambda}e^{-s\Phi(\lambda)},
\end{equation*}
where $\cL_{t \to \lambda}$ denotes the Laplace transform operator.\\
A particular case is given by the choice $\Phi(\lambda)=\lambda^\alpha$ for $\alpha \in (0,1)$. Indeed, in this case, we have the $\alpha$-stable subordinator $\sigma_\alpha(t)$. According  to  \cite{meerschaert2013inverse}, $\sigma_\alpha(t)$ and $E_\alpha(t)$ are absolutely continuous random variables for any $t>0$ and, if we denote by $g_\alpha(x)$ the probability density function of $\sigma_\alpha(1)$, it was proved in \cite{meerschaert2013inverse} that
\begin{equation}\label{eqftog}
f_\alpha(x,t)=\frac{t}{\alpha}x^{-1-\frac{1}{\alpha}}g_\alpha(tx^{-\frac{1}{\alpha}}).
\end{equation}
The following lemma presents the result whose  proof is contained in  Remarks $4.2$ and $4.4$ of \cite{ascione2019time}, therefore now it is omitted.
\begin{lem}\label{lem:mominv}
	It holds that   $\E[E_\alpha^{-\gamma}(t)]<+\infty$ for any $t>0$ and any $\gamma \in \left(0,1\right)$.   In contrast, for any $n>1$ and $H \in \left(\frac{1}{2},1\right)$ the higher-order moments are infinite:  $\E[E_\alpha^{-nH}(t)]=+\infty$.
\end{lem}
It is true in general that $\E[E_\Phi^{-\gamma}(t)]<+\infty$ for any $\gamma \in \left(0,1\right)$ . Indeed, we have the following result.
\begin{lem}\label{lem:mom1}
	Let $\gamma \in (0,1),\;t>0$. Then $\E[E_\Phi^{-\gamma}(t)]<+\infty$.
\end{lem}
\begin{proof}
Let us recall that $\lim_{s \to 0^+}f_\Phi(s;t)=\bar{\nu}_\Phi(t):=\nu_\Phi(t,+\infty)$, that is finite since $\nu_\Phi$ is a L\'evy measure. Moreover, $s \mapsto f_\Phi(s;t)$ is continuous, thus $\Norm{f_\Phi(\cdot,t)}{L^\infty(0,1)}<+\infty$. Hence it holds
\begin{align*}
\E[E_\Phi^{-\gamma}(t)]&=\int_0^{+\infty}s^{-\gamma}f_\Phi(s;t)ds\\&=\int_0^{1}s^{-\gamma}f_\Phi(s;t)ds+\int_1^{+\infty}s^{-\gamma}f_\Phi(s;t)ds\\
&\le \Norm{f_\Phi(\cdot,t)}{L^\infty(0,1)}\int_0^{1}s^{-\gamma}ds+1<+\infty.
\end{align*}
\end{proof}
\section{The time-changed fractional Ornstein-Uhlenbeck process}\label{Sec2}
Let $(\Omega, \mathcal{F}, P)$ be a complete probability space supporting all \textcolor{black}{the} stochastic processes that will be considered below. Let us fix Hurst index $H \in \left(\frac{1}{2},1\right)$  and consider a fractional Brownian motion $B^H=\{B^H(t), t\ge 0\}$ with Hurst index $H$, i.e. a centered Gaussian process with covariance function given by $$\E[B^H(t)B^H(s)]=1/2\left(t^{2H}+s^{2H}-|t-s|^{2H}\right),\; s,t \in \mathbb{R}^+.$$
Let us also fix some number $\theta>0$ and introduce the fractional Ornstein-Uhlenbeck process (defined in \cite{cheridito2003fractional}) as
\begin{equation}\label{UHdef}
U_H(t)=e^{-\frac{t}{\theta}}\int_0^{t}e^{\frac{s}{\theta}}dB^H(s),\qquad  t\ge 0,
\end{equation}
which is a Gaussian process with one-dimensional density
\begin{equation}\label{pH}
p_H(t,x)=\frac{1}{\sqrt{2\pi V_{2,H}(t)}}e^{-\frac{x^2}{2V_{2,H}(t)}},  \ x \in \R, \ t \ge 0,
\end{equation}
where
$$V_{2,H}(t)=e^{-2\frac{t}{\theta}}\int_0^{t}\int_0^{t}e^{\frac{v+u}{\theta}}|u-v|^{2H-2}dudv, t\ge 0$$
is the variance of $U_H(t)$. Let us also denote by $V_{n,H}(t)=\E[|U_H(t)|^n]$ the $n$th absolute moment of $U_H(t)$.\\
Now, as done in \cite{ascione2019time}, we can construct the time-changed fractional Ornstein-Uhlenbeck process by considering a fractional Ornstein-Uhlenbeck process $U_H(t)$, together with  an independent inverse subordinator $E_\Phi(t)$, and   defining
\begin{equation*}
U_{H,\Phi}(t):=U_H(E_\Phi(t)).
\end{equation*}
Let us state some properties of $V_{n,H,\Phi}(t):=\E[|U_{H,\Phi}(t)|^{n}]$ (see \cite[Lemma $3.1$]{ascione2019time}).
\begin{prop}\label{prop:varprop}
	\begin{itemize}
		\item [$(i)$] $V_{2n,H,\Phi}(t)$ is finite for any $t>0$ and $n \in \N$.
		\item [$(ii)$] It holds that
		\begin{equation*}
		V_{2n,H,\Phi}(t)=\int_0^{+\infty}V_{2n,H}(s)f_\alpha(s,t)ds.
		\end{equation*}
		\item [$(iii)$] $V_{2n,H,\Phi}(t)$ is increasing in $t$ for any $n \in \N$ and
		\begin{equation*}
		\lim_{t \to +\infty}V_{2n,H,\Phi}(t)=V_{2n,H}(\infty)=\frac{\left(2\theta^{2H}H\Gamma(2H)\right)^n\Gamma\left(\frac{2n+1}{2}\right)}{\sqrt{\pi}}.
		\end{equation*}
	\end{itemize}
\end{prop}
\begin{rmk} The fact that the asymptotic value in $(iii)$ does not depend on $\Phi$ is strictly connected to the nature of the time-change. Indeed, $E_\Phi(t)$ acts as a delay in the time-scale of $U_H(t)$, hence we expect $U_{H,\Phi}(t)$ to have the same asymptotic behaviour, despite behaving quite differently on the whole trajectories.
\end{rmk}
We will also need the following limits  for $V_{2,H}(t)$ and its derivative. They can  be obtained by \cite[Equation $29$]{ascione2019fractional} and \cite[Lemma $5.2$]{ascione2019time}.
\begin{lem}\label{lem:Varlem}
	Function $V_{2,H}$ satisfies the relations $$\lim_{t \to 0^+}\frac{V_{2,H}(t)}{t^{2H}}=1\; \text{and} \; \lim_{t \to +\infty}V_{2,H}(t)=\theta^{2H}H\Gamma(2H).$$ Moreover, $V_{2,H} \in C^1(0,+\infty)$ and its derivative satisfies the relations
	\begin{align*}
	\lim_{t \to 0^+}\frac{V'_{2,H}(t)}{t^{2H-1}}=2H\;  \text{and}\;\lim_{t \to +\infty}e^{\frac{t}{\theta}}t^{2-2H}V'_{2,H}(t)=2H(2H-1)\theta.
	\end{align*}
\end{lem}
In the next subsection we will consider some preliminary convergence results concerning the fractional Ornstein-Uhlenbeck process.
\subsection{Limit behaviour of the marginals of the fractional Ornstein-Uhlenbeck as $H \to 1/2$}
Now we want to study the one-dimensional convergence of $U_H(t)$ to the (classical) Ornstein-Uhlenbeck $U(t)$ as $H \to 1/2^+$. To do this, let us first study the convergence of the mean and the variance. Actually, since we asked for $U_H(0)=0$, we have $\E[U_H(t)]=0$ by equation \eqref{UHdef}. Thus we have to study the convergence of the variance. However, even more generally, let us study the convergence of all the moments.
\begin{prop}\label{prop:convmom}
	For any $n \ge 1$ it holds that $\lim_{H \to \frac{1}{2}^+}V_{n,H}(t)=V_{n,\frac{1}{2}}(t)$ \textcolor{black}{and} the convergence is uniform in $[0,+\infty)$.
\end{prop}
\begin{proof}
	Let us first study the uniform convergence of $V_{2,H}$. Let us recall that, as shown in \cite{ascione2019fractional}, the function $(H,t) \in \left[\frac{1}{2},1\right)\times [0,+\infty) \to V_{2,H}(t)$ is continuous. Moreover, we can extend it by continuity to $[0,+\infty]$ by setting $V_{2,H}(+\infty)=\theta^{2H}H\Gamma(2H)$ and then we have that $(H,t) \in \left[\frac{1}{2},\frac{3}{4}\right]\times [0,+\infty] \to V_{2,H}(t)$ is uniformly continuous by Cantor's theorem. Uniform convergence as $H \to 1/2$ follows from this last property.\\
	Now recall that  $U_H(t)$, being a Gaussian process, admits the  following equalities for its absolute moments:
	\begin{equation}\label{eq:pass0}
	V_{n,H}(t)=\frac{2^{\frac{n}{2}}\Gamma\left(\frac{n+1}{2}\right)}{\sqrt{\pi}}(V_{2,H}(t))^\frac{n}{2}.
	\end{equation}
	Hence, we have that
	\begin{equation}\label{eq:pass1}
	 V_{n,H}(t)-V_{n,\frac{1}{2}}(t)=\frac{2^{\frac{n}{2}}\Gamma\left(\frac{n+1}{2}\right)}{\sqrt{\pi}}((V_{2,H}(t))^\frac{n}{2}-(V_{2,\frac{1}{2}}(t))^\frac{n}{2}).
	\end{equation}
	Let us consider $H \in \left[\frac{1}{2},\frac{3}{4}\right]$ and fix $V=\max_{H \in \left[\frac{1}{2},\frac{3}{4}\right]}V_{2,H}(\infty)$. For $n \ge 2$ set $L_1(n)=\frac{n}{2}V^{\frac{n}{2}-1}$ and observe that
	\begin{equation*}
	|(V_{2,H}(t))^\frac{n}{2}-(V_{2,\frac{1}{2}}(t))^\frac{n}{2}|\le L(n)|V_{2,H}(t)-V_{2,\frac{1}{2}}(t)| \le L_1(n)\Norm{V_{2,H}-V_{2,\frac{1}{2}}}{L^\infty(0,+\infty)},
	\end{equation*}
	where the term at the right-hand side is finite since both $V_{2,H}$ and $V_{2,\frac{1}{2}}$ are bounded.\\
	Taking the supremum in equation \eqref{eq:pass1}, we have
	\begin{equation*}
	\Norm{V_{n,H}-V_{n,\frac{1}{2}}}{L^\infty(0,+\infty)}\le L_2(n) \Norm{V_{2,H}-V_{2,\frac{1}{2}}}{L^\infty(0,+\infty)},
	\end{equation*}
	where $L_2(n)=\frac{2^{\frac{n}{2}}\Gamma\left(\frac{n+1}{2}\right)}{\sqrt{\pi}}L_1(n)$, and then we can take the limit as $H \to \frac{1}{2}^+$ to conclude that, for any $n \ge 2$,
	\begin{equation*}
	\lim_{H \to \frac{1}{2}^+}\Norm{V_{n,H}-V_{n,\frac{1}{2}}}{L^\infty(0,+\infty)}=0.
	\end{equation*}
	The case $n=1$ is different. Indeed, in this case, equation \eqref{eq:pass0} implies that $V_{1,H}(t)=\left(\frac{2}{\pi} V_{2,H}(t)\right)^{1/2}$. However, by subadditivity of $t \mapsto \sqrt{t}$, we have
	\begin{equation*}\begin{gathered}
	|V_{1,H}(t)-V_{1,\frac{1}{2}}(t)|\le \left(\frac{2}{\pi}\right)^{1/2}\left|\left(V_{2,H}(t)\right)^{1/2}-\left(V_{2,\frac{1}{2}}(t)\right)^{1/2}\right|\\\le \left(\frac{2}{\pi} |V_{2,H}(t)-V_{2,\frac{1}{2}}(t)|\right)^{1/2} 
	 \le \left(\frac{2}{\pi} \Norm{V_{2,H}-V_{2,\frac{1}{2}}}{L^\infty(0,+\infty)}\right)^{1/2}.
	\end{gathered}\end{equation*}
	Now, taking the supremum, we obtain
	\begin{equation*}
	\Norm{V_{1,H}-V_{1,\frac{1}{2}}}{L^\infty(0,+\infty)}\le \left(\frac{2}{\pi} \Norm{V_{2,H}-V_{2,\frac{1}{2}}}{L^\infty(0,+\infty)}\right)^{1/2}.
	\end{equation*}
	Taking the limit as $H \to \frac{1}{2}^+$ we finally achieve
	\begin{equation*}
	\lim_{H \to \frac{1}{2}^+}\Norm{V_{1,H}-V_{1,\frac{1}{2}}}{L^\infty(0,+\infty)}=0.
	\end{equation*}
\end{proof}
Now let us observe that, since $\E[U_H(t)]=0$, $V_{2,H}(t)\to V_{2,\frac{1}{2}}(t)$, and $U_H(t)$ is a Gaussian random variable, we have that, for fixed $t>0$, $U_H(t) \overset{d}{\to}U_{\frac{1}{2}}(t)$. Thus we already have the weak one-dimensional convergence of the fractional Ornstein-Uhlenbeck process to the classical one as $H \to 1/2^+$. However, we can improve this kind of convergence, showing that the density $p_H(x,t)$ converges uniformly to $p_{1/2}(x,t)$.
\begin{thm}\label{convp}
	It holds that
	$$\lim_{H \to \frac{1}{2}^+}p_{H}(t,x)=p_{\frac{1}{2}}(t,x)$$
	for any $t \in (0,+\infty)$ and for any $x \in \R$. Moreover, for any compact set $K \subset \R\setminus\{0\}$, $p_{H,\alpha} \to p_{\frac{1}{2},\alpha}$ uniformly in $[0,+\infty)\times K$ as $H \to \frac{1}{2}$.
\end{thm}
\begin{proof}
	We have already discussed the pointwise convergence of the densities, thus let us show the uniform convergence. Here, we suppose, without loss of generality, that $H \le \frac{3}{4}$. Consider $K \subset \R\setminus\{0\}$ a compact set. Let us suppose, also without loss of generality, that $K=[a,b]$ for some $a,b>0$. Define the function
	\begin{equation}\label{Heatkernel}
	p(t,x)=\frac{1}{\sqrt{2\pi t}}e^{-\frac{x^2}{2t}}
	\end{equation}
	for $(t,x)\in (0,+\infty)\times K $. We can extend it by continuity to $(t,x)\in [0,+\infty) \times K$ by setting $p(0,x)=0$. Now, differentiating $p$ with respect to $x$ and $t$, we obtain the equalities
	\begin{align}\label{eq:pass2}
	\pd{p}{t}(t,x)=\frac{1}{2\sqrt{2\pi}t^{\frac{3}{2}}}e^{-\frac{x^2}{2t}}\left(\frac{x^2-t}{t}\right) && \pd{p}{x}(t,x)=\frac{-x}{\sqrt{2\pi}t^{\frac{3}{2}}}e^{-\frac{x^2}{2t}}.
	\end{align}
	Let us consider the vector
	\begin{equation*}
	q(t,x)=\left(\frac{x^2-t}{t},-2x\right)
	\end{equation*}
	to obtain from equation \eqref{eq:pass2} that $\nabla p(t,x)=\frac{1}{2\sqrt{2\pi}t^{\frac{3}{2}}}e^{-\frac{x^2}{2t}}q(t,x)$. Obviously,
	\begin{equation*}
	|q(t,x)|=\frac{\sqrt{(x^2-t)^2+4t^2x^2}}{t},
	\end{equation*}
	and therefore
	\begin{equation*}
	|\nabla p(t,x)|=\frac{\sqrt{(x^2-t)^2+4t^2x^2}}{2\sqrt{2\pi}t^{\frac{5}{2}}}e^{-\frac{x^2}{2t}}.
	\end{equation*}
	Observe that we can extend   $|\nabla p(t,x)|$ by continuity to $[0,+\infty)\times K$  by setting $$|\nabla p(0,x)|=0.$$
	Let us now fix a compact set $K_2=[0,T]$ for some $T>0$ and let $(t,x),(s,y) \in K_2 \times K$. Since $K_2 \times K$ admits convex interior, we can apply Lagrange's Theorem to show that there exists a point $(\tau,z) \in [(t,x),(s,y)]$ (where $[(t,x),(s,y)]$ is the segment connecting $(t,x)$ to $(s,y)$) such that
	\begin{equation*}
	p(t,x)-p(s,y)=\langle \nabla p(\tau,z),(t-s,x-y)\rangle.
	\end{equation*}
	Taking the absolute value and using Cauchy-Schwartz inequality, we get that
	\begin{equation}\label{eq:pass4}
	|p(t,x)-p(s,y)|\le  |\nabla p(z,\tau)||(t-s,x-y)|.
	\end{equation}
	Finally, since we have verified that $|\nabla p(t,x)|$ is continuous in $K_2 \times K$ (that is compact), we can take the maximum, achieving
	\begin{equation*}
	|p(t,x)-p(s,y)|\le (\max_{(\tau,z)\in K_2 \times K}|\nabla p(\tau,z)|)|(t-s,x-y)|.
	\end{equation*}
	Now let us observe that the function $H \mapsto V_{2,H}(+\infty)$ is   continuous  on the interval  $\left[\frac{1}{2},\frac{3}{4}\right]$ for any $\theta>0$. Therefore,  we can introduce finite values  $$T=\max_{H \in \left[\frac{1}{2},\frac{3}{4}\right]}V_{2,H}(+\infty)\;\;\text{and}\;\;  C_5(K)=\max_{(\tau,z)\in K_2 \times K}|\nabla p(\tau,z)|,$$ accompanied by the compact set   $K_2=[0,T]$.  Observe also that, according to  \eqref{Heatkernel} and \eqref{pH}, $$p(V_{2,H}(t),x)=p_{H}(t,x).$$ Together with inequality   \eqref{eq:pass4}, it means that
	\textcolor{black}{\begin{align*}
		|p_H(t,x)-p_{\frac{1}{2}}(t,x)|&=|p(V_{2,H}(t),x)-p(V_{2,\frac{1}{2}}(t),x)|\\
		&\le
		C_5(K)|V_{2,H}(t)-V_{2,\frac{1}{2}}(t)|\\
		&\le C_5(K)\Norm{V_{2,H}-V_{2,\frac{1}{2}}}{L^\infty(0,+\infty)}.
		\end{align*}}
	Taking the supremum as $(t,x)\in [0,+\infty)\times K$ we obtain
	\begin{equation*}
	\Norm{p_H-p_{\frac{1}{2}}}{L^\infty([0,+\infty)\times K)}\le C_5(K)\Norm{V_{2,H}-V_{2,\frac{1}{2}}}{L^\infty(0,+\infty)}.
	\end{equation*}
	Finally, taking the limit as $H \to 1/2^+$, we conclude the proof.
\end{proof}
Now,  having exploited the main convergence results on the fractional Ornstein-Uhlenbeck process, we can focus on the convergence results for the time-changed one.
\section{The limit distribution of the TCfOU process}\label{Sec3}
Let us first explore the limit distribution of the TCfOU process as $t \to +\infty$. It is well known that (for negative drift parameters) the fractional Ornstein-Uhlenbeck process is ergodic and admits Gaussian limit distribution with density function
\begin{equation*}
p_H(\infty,x)=\frac{1}{\sqrt{2\pi V_H}}e^{-\frac{x^2}{2V_H}}
\end{equation*}
where $V_H:=V_{2,H}(\infty)$. By using this result we can exploit the limit distribution of the TCfOU process.
\begin{prop}
	Let $p_{H,\Phi}(t,x)$ be the probability density function of $U_{H,\Phi}(t)$ (that exists by \cite[Proposition $4.1$]{ascione2019time}). Then it holds that
	\begin{equation}\label{convedens}
	\lim_{t \to +\infty}p_{H,\Phi}(t,x)=p_H(\infty,x).
	\end{equation}
	Moreover, $U_{H,\Phi}(t)\overset{d}{\to}Z$ as $t \to +\infty$, where $Z \sim \cN(0,V_H)$.
\end{prop}
\begin{proof}
	Let us observe that if we are able to establish  equality \eqref{convedens}, then the weak convergence result directly follows (since the convergence of the densities implies the weak convergence).\\
	It follows  from \cite[Proposition $4.1$]{ascione2019time}  that 
	\begin{equation}\label{eq:pass5}
	p_{H,\Phi}(t,x)=\int_0^{+\infty}p_H(s,x)f_\Phi(s;t)ds=\E[p_H(E_\Phi(t),x)],
	\end{equation}
	where $p_H$ is defined by equality \eqref{pH}.\\
	Now fix $x \not = 0$. Then we have that $p_{H}(t,x)\le C(x,H)$. Thus, recalling that $\lim_{t \to +\infty}E_\Phi(t)=+\infty$ almost surely, we have by dominated convergence theorem
	\begin{equation*}
	\lim_{t \to +\infty}p_{H,\Phi}(t,x)=\E[p_H(x,\infty)]=p_H(x,\infty).
	\end{equation*}
	Now let us consider $x=0$. Then we have, still from equality \eqref{pH}, that
	\begin{equation*}
	p_{H}(t,0)=C(H)\frac{1}{\sqrt{V_{2,H}(t)}}
	\end{equation*}
	We only need to show that $\E\left[\frac{1}{\sqrt{V_{2,H}(E_\Phi(t))}}\right]<+\infty$. Let us start with the following upper bounds: 
	\begin{align*}
	\E\left[\frac{1}{\sqrt{V_{2,H}(E_\Phi(t))}}\right]&=\int_0^{+\infty}\frac{1}{\sqrt{V_{2,H}(s)}}f_\Phi(s;t)ds\\
	&=\int_0^{1}\frac{1}{\sqrt{V_{2,H}(s)}}f_\Phi(s;t)ds+\int_1^{+\infty}\frac{1}{\sqrt{V_{2,H}(s)}}f_\Phi(s;t)ds\\
	&=I_1+I_2.
	\end{align*}
	Concerning $I_2$, obviously $V_{2,H}(s)\ge V_{2,H}(1)$ as $s \ge 1$, therefore  
	\begin{equation*}
	I_2\le \frac{1}{\sqrt{V_{2,H}(1)}}.
	\end{equation*}
	Concerning $I_1$, since $\lim_{t \to 0^+}\frac{V_{2,H}(t)}{t^{2H}}=1$ by Lemma \ref{lem:Varlem}, there exists a constant $C(H)$ such that $V_{2,H}(t)\ge C(H) t^{2H}$ for any $t \in [0,1]$. Thus it follows from Lemma \ref{lem:mom1} that
	\begin{equation*}
	I_1 \le C(H)\E[E_\Phi^{-H}(t)]<+\infty.
	\end{equation*}
	We get that $\E\left[\frac{1}{\sqrt{V_{2,H}(E_\Phi(t))}}\right]<+\infty$ and then we can use dominated convergence theorem to conclude that limit relation  \eqref{convedens} holds even if $x=0$.
	\end{proof}
\begin{rmk}
	This result was expected. Indeed, the action of the time-change consists in change of the time scale that generally does not affect the limit distributions.
\end{rmk}
\section{Weak one-dimensional convergence of the TCfOU process and uniform convergence of the marginals}\label{Sec4}
Let us denote by $U(t)$ a classical Ornstein-Uhlenbeck process and by $U_\Phi(t):=U(E_\Phi(t))$ a time-changed Ornstein-Uhlenbeck process, where $E_\Phi(t)$ is independent of $U(t)$. This kind of process has been introduced in \cite{leonenko2013fractional} for the stable case and \cite{gajda2015time} for the general case. Here we want to discuss the one-dimensional convergence of the TCfOU process to the time-changed Ornstein-Uhlenbeck process. To do this, let us actually demonstrate some more strong convergence of the probability density function.
\begin{thm}\label{thm:conv}
	Let $p_{H,\Phi}(t,x)$ be the probability density function of $U_{H,\Phi}(t)$ and $p_{\frac{1}{2},\Phi}(t,x)$ the probability density function of $U_\Phi(t)$. Then it holds
	\begin{equation}\label{convH12}
	\lim_{H \to \frac{1}{2}^+}p_{H,\Phi}(t,x)=p_{\frac{1}{2},\Phi}(t,x)
	\end{equation}
	for any $t>0$ and $x \in \R$. Moreover, for any compact set $K \subset \R\setminus\{0\}$ it holds $p_{H,\Phi}\to p_{\frac{1}{2},\Phi}$ uniformly in $[0,+\infty) \times K$ as $H \to \frac{1}{2}^+$.
\end{thm}
\begin{proof}
	It follows from  equality  \eqref{eq:pass5} that for any $H \in \left[\frac{1}{2},1\right)$
	\begin{equation*}
	p_{H,\Phi}(t,x)=\int_0^{+\infty}p_{H}(s,x)f_\Phi(s;t)ds.
	\end{equation*}
	We have $p_H(s,x)\le \frac{1}{\sqrt{2\pi V_{2,H}(s)}}$. Consider $H \in \left[\frac{1}{2},\frac{3}{4}\right]$ and let us split the integral in two parts:
	\begin{align*}
	p_{H,\Phi}(t,x)&=\int_0^{+\infty}p_{H}(s,x)f_\Phi(s;t)ds\\
	&\le \frac{1}{\sqrt{2\pi}}\left(\int_0^1 \frac{1}{\sqrt{ V_{2,H}(s)}}f_\Phi(s;t)ds+\int_1^{+\infty} \frac{1}{\sqrt{ V_{2,H}(s)}}f_\Phi(s;t)ds\right)\\
	&=\frac{1}{\sqrt{2\pi}}(I_1(H)+I_2(H)).
	\end{align*}
	Let us denote by $V_1=\min_{H \in \left[\frac{1}{2},\frac{3}{4}\right]}V_{2,H}(1)$ to achieve $V_{2,H}(s)\ge V$ for any $s \ge 1$, thus dominating the integrand of $I_2(H)$. Concerning $I_1(H)$, we have
	\begin{equation*}
	I_1(H)=\int_0^1 \frac{s^{H}}{\sqrt{V_{2,H}(s)}}s^{-H}f_\Phi(s;t)ds.
	\end{equation*}
	Now let us observe that for $s \in (0,1)$ we have that $s^{-H}\le s^{-3/4}$. On the other hand, the function $(H,s)\in \left[\frac{1}{2},\frac{3}{4}\right]\times (0,1] \mapsto \frac{s^{H}}{\sqrt{V_{2,H}(s)}}$ is continuous and positive. Moreover, it can be extended by continuity setting $(H,0)\mapsto 1$. Thus the aforementioned function admits a maximum $V_2>0$ and we deduce that
	\begin{equation*}
	\frac{s^{H}}{\sqrt{V_{2,H}(s)}}s^{-H}\le V_2 s^{-3/4}.
	\end{equation*}
	In conclusion, we obtained the following upper bounds  for any $H \in \left[\frac{1}{2},\frac{3}{4}\right]$:
	\begin{equation*}
	p_{H}(s,x)\le \frac{1}{\sqrt{2\pi}}\begin{cases} 1/\sqrt{V_1} & s>1 \\
	  V_2 s^{-3/4} & 0 \le s \le 1,
	  \end{cases}
	 \end{equation*}
	 that is integrable with respect to $f_\Phi(s;t)ds$ since $\E[E_\Phi^{-3/4}(t)]<+\infty$ by Lemma \ref{lem:mom1}.
	 Thus, by dominated convergence theorem, we achieve equation \eqref{convH12} by taking the limit as $t \to +\infty$ in equation \eqref{eq:pass5}.\\
	 Concerning uniform convergence, let us fix a compact $K \subseteq \R \setminus\{0\}$. Moreover, still by using equation \eqref{eq:pass5}, we have
	 \begin{equation*}
	 |p_{H,\Phi}(t,x)-p_{\frac{1}{2},\Phi}(t,x)|\le \Norm{p_{H}-p_{\frac{1}{2}}}{L^\infty((0,+\infty)\times K)}
	 \end{equation*}
	 and, taking the supremum,
	 \begin{equation*}
	 \Norm{p_{H,\Phi}-p_{\frac{1}{2},\Phi}}{L^\infty((0,+\infty) \times K)}\le \Norm{p_{H}-p_{\frac{1}{2}}}{L^\infty((0,+\infty) \times K)}.
	 \end{equation*}
	 Thus, taking the limit as $H \to \frac{1}{2}^+$, we conclude the proof by using Theorem \ref{convp}.
\end{proof}
As for the fractional Ornstein-Uhlenbeck process, we can establish  also the uniform convergence of the absolute moments.
\begin{prop}
	It holds that $V_{n,H,\Phi}\to V_{n,\frac{1}{2},\Phi}$ as $H \to \frac{1}{2}^+$ uniformly in $[0,+\infty)$.
\end{prop}
\begin{proof}
	Consider $H \in \left[\frac{1}{2},\frac{3}{4}\right]$. Let us recall that, by Proposition \ref{prop:varprop},
	\begin{equation*}
	V_{n,H,\Phi}(t)=\int_0^{+\infty}V_{n,H}(s)f_\Phi(s;t)
	\end{equation*}
	and then
	\begin{equation*}
	|V_{n,H,\Phi}(t)-V_{n,\frac{1}{2},\Phi}(t)|\le \Norm{V_{n,H}-V_{n,\frac{1}{2}}}{L^\infty(0,+\infty)}.
	\end{equation*}
	Taking the supremum as $t \in (0,+\infty)$ and then the limit as $H \to \frac{1}{2}^+$ we conclude the proof by means of Proposition \ref{prop:convmom}.
\end{proof}
Actually, we are not happy with the one-dimensional convergence. Hence, in the next section, we want to show a functional limit theorem for the TCfOU process as $H \to 1/2$.
\section{A functional limit theorem for the TCfOU process}\label{Sec5}
Let us consider the function space $C=C(0,+\infty)$ equipped with the uniform norm. Let $U_{H,\Phi}=\{U_{H,\Phi}(t), \ t \ge 0\}$ and $U_{\Phi}=\{U_{\Phi}(t), \ t \ge 0\}$. Let us also set $U_H=\{U_{H}(t), \ t \ge 0\}$ and $U=\{U(t), \ t \ge 0\}$. All these four processes can be seen as $C$-valued random variables and then we can study the weak convergence in $C$ of such random variables. In particular, let us denote by $C^*$ the dual space of $C$ (i.e. the space of continuous linear functionals on $C$). Then we say that, for a sequence of $C$-valued random variables $(X_n)_{n \in \N}$ and a $C$-valued random variable $X \in C$, it holds $X_n \Rightarrow X$ if and only if for any $\cF \in C^*$ it holds $\E[\cF(X_n)] \to \E[\cF(X)]$ as $n \to +\infty$. Let us observe that such definition is actually valid for any $B$-valued sequence of random variables, where $B$ is a metric space and $B^*$ is its dual. In particular, for any $\gamma \in (0,1]$, let us denote by ${\rm Lip}_\gamma([0,T])$ the space of H\"older-continuous functions on $[0,T]$ with exponent $\gamma$, which is a Banach space when equipped with the norm
\begin{equation*}
\Norm{f}{{\rm Lip}_\gamma([0,T])}=\sup_{\substack{(t,s) \in [0,T]^2\\ t \not = s}}\frac{|f(t)-f(s)|}{|t-s|^\gamma}+\Norm{f}{L^\infty(0,T)}.
\end{equation*}
Let us focus on $U_H$ and $U$. In this connection, let us first consider $t \ge s \ge 0$ and define $R_H(t,s)=\Cov(U_H(t),U_H(s))$. We can state that  (see \cite{mishura2018stochastic})
\begin{equation}\label{covcov1}
\begin{gathered}
R_H(t,s)=\frac{H\sigma^2}{2}\left(-e^{\frac{ s-t}{\theta}}\int_0^{t-s}e^{\frac{z}{\theta}}
z^{2H-1}dz+e^{\frac{t-s}{\theta}}\int_{t-s}^{t}e^{-\frac{z}{\theta}}z^{2H-1}dz\right.\\\left. -e^{-\frac{t+s}{\theta}}\int_{s}^{t}e^{\frac{z}{\theta}}z^{2H-1}dz+e^{ \frac{ s-t}{\theta}}
\int_{0}^{s}e^{-\frac{z}{\theta}}z^{2H-1}dz +2e^{-\frac{t+s}{\theta}}\int_0^{t}e^{\frac{z}{\theta}}z^{2H-1}dz\right).
\end{gathered}
\end{equation}
For any $t,s \in (0,+\infty)$, set $C_H(t,s)=\Cov(U_H(t),U_H(s))$. In particular we have
\begin{equation*}
	C_H(t,s)=\begin{cases} R_H(t,s) & t \ge s \ge 0 \\
	R_H(s,t) & 0 \le t<s. \end{cases}
\end{equation*}
As $t \not = s$ it is not difficult to check that $C_H(t,s)$ admits both partial derivatives and that such partial derivatives are continuous. Moreover, we have $R_H(t,s) \to 0$ as $t \to +\infty$, hence $\pd{C_H}{t}(t,s)\to 0$ for any $s \in (0,+\infty)$. The same holds inverting the roles of $t$ and $s$, since $C_H$ is symmetric. Finally, as $t \to s^-$, we have that both $\pd{}{t}R_H(t,s)$ and $\pd{}{s}R_H(t,s)$ can be extended by continuity. However, such functions are not continuous also with respect to $H \in \left[\frac{1}{2},\frac{3}{4}\right]$. However, it can be still shown by simple, but cumbersome, calculations that there exists a constant independent of $H$ such that
\begin{equation}\label{Lipschitzcontr}
\Norm{\pd{}{t}C_H(t,s)}{L^\infty((0,+\infty)\times (0,+\infty))}+\Norm{\pd{}{s}C_H(t,s)}{L^\infty((0,+\infty)\times (0,+\infty))}\le L
\end{equation}
for any $H \in \left[\frac{1}{2},\frac{3}{4}\right]$. This is enough to show the following result.
\begin{thm}
	It holds $U_H \Rightarrow U$ in $C$ as $H \to \frac{1}{2}^+$. Moreover, for any $T>0$ and any linear functional $\cF \in {\rm Lip_{\gamma}}^*([0,T])$ for $\gamma<1/2$ it holds $\E[\cF(U_H)] \to \E[\cF(U)]$.
\end{thm}
\begin{proof}
	Let us observe that $U_H$ is a Gaussian process with zero mean and covariance function $C_H(t,s)$ that is continuous for $H \in \left[\frac{1}{2},\frac{3}{4}\right] \times [0,+\infty)\times [0,+\infty)$. Thus we know that $U_H$ converges toward $U$ in any $d$-dimensional distributions. Now let us recall that $U_H(t)-U_H(s)$ is a Gaussian random variable whose variance is given by $(C_H(t,t)-2C_H(t,s)+C_H(s,s))$. Thus we have that  for any $n \ge 1$ the next relations follow from  equality  \eqref{eq:pass0}:
	\begin{equation}\label{eq:pass7}
	\E[(U_H(t)-U_H(s))^{2n}]=\frac{2^n\Gamma\left(n+\frac{1}{2}\right)}{\sqrt{\pi}}(C_H(t,t)-2C_H(t,s)+C_H(s,s))^n.
	\end{equation}
	Let us suppose, without loss of generality, that $s \le t$. There exist two constants $\xi,\eta \in [s,t]$ such that, by Lagrange theorem,
	\begin{equation*}
	C_H(t,t)-2C_H(t,s)+C_H(s,s)=\left(\pd{}{t}C_H(t,\xi)+\pd{}{s}C_H(\eta,s)\right)(t-s).
	\end{equation*}
	Taking the absolute value, we get the inequalities 
	\begin{equation*}\begin{gathered}
	|C_H(t,t)-2C_H(t,s)+C_H(s,s)|\\ \le \left(\left|\pd{}{t}C_H(t,\xi)\right|+\left|\pd{}{s}C_H(\eta,s)\right|\right)|t-s| \le L|t-s|,
	\end{gathered}\end{equation*}
	where the constant $L$ is defined according to  \eqref{Lipschitzcontr}.
	Denoting
	\begin{equation*}
	C_n(L)=\frac{(2L)^n\Gamma\left(n+\frac{1}{2}\right)}{\sqrt{\pi}},
	\end{equation*}
	we obtain from equality \eqref{eq:pass7} that 
	\begin{equation*}
	\E[(U_H(t)-U_H(s))^{2n}]\le C_n(L)|t-s|^n.
	\end{equation*}
	Hence, in particular, by Kolmogorov's tightness criterion (that can be applied for a non-bounded time interval by means of a truncation argument, as shown in \cite{totoki1962method}), we can use Prokhorov'theorem to conclude that $U_H \Rightarrow U$ in $C$.\\
	Moreover, for any fixed $T>0$, we are also under the hypotheses of \cite[Corollary $2$]{lamperti1962convergence}, that ensures the weak convergence in ${\rm Lip}_\gamma$, concluding the proof.
\end{proof}
\begin{rmk}
The bound $\gamma<1/2$ is sharp since $U$ belongs to ${\rm Lip_{\gamma}}([0,T])$ for $\gamma<1/2$ but not for $\gamma=1/2$.
\end{rmk}
Now we need to provide  a similar result for $U_{H,\Phi}$ and $U_{\Phi}$. However, we have to consider a different space. In particular, let us consider the set $D$ of cadlag functions on $[0,+\infty)$ and $\Lambda$ the set of strictly increasing functions $g:[0,+\infty) \to [0,+\infty)$. Let $\iota \in \Lambda$ be the identity function on $[0,+\infty)$. Then we can define the Skorohod $J_1$ metric on $D$ as, for any couple of functions $f_1,f_2 \in D$,
\begin{equation*}
d_{J_1}(f_1,f_2)=\inf_{g \in \Lambda}\max\{\Norm{f_1 \circ g-f_2}{L^\infty(0,+\infty)},\Norm{g-\iota}{L^\infty(0,+\infty)}\}.
\end{equation*}
With this metric, the set $(D,d_{J_1})$ is a metric space (that we will denote only as $D$), thus we can consider the notion of weak convergence of $D$-valued random variables, which is actually weaker than the weak convergence in $C$.\\
With this in mind, let us show the following result.
\begin{thm}
	It holds $U_{H,\Phi} \Rightarrow U_\Phi$ in $D$ as $H \to \frac{1}{2}^+$.
\end{thm}
\begin{proof}
	Let us first observe that $C$ is a closed subspace of $D$, hence $D^*$ is contained in $C^*$. Thus the fact that $U_{H}\Rightarrow U$ in $C$ implies the same convergence in $D$. Now let us consider the coupled processes $(U_H,L_\Phi)$ and $(U,L_\Phi)$. Since $L_\Phi$ is independent of both $U_H$ and $U$, we get $(U_H,L_\Phi) \Rightarrow (U,L_\Phi)$ in $D \times D$. Now let us observe that $U \in C$ almost surely and $L_\Phi \in D_{\uparrow}$ almost surely, where $D_{\uparrow}=\{f \in D: \ f\mbox{ is increasing}\}$. Now let us denote by $g$ the composition map on $D \times D$, i.e. $g(x,y)=x \circ y \in D$, and with ${\rm Disc}(g)$ the set of discontinuity points of $g$. \\
	It follows from  \cite[Theorem $13.2.2$]{whitt2002stochastic} (see also \cite{silvestrov2012limit}  for a survey on continuity conditions for the composition map) that $${\rm Disc}(g)\subseteq (D \times D)\setminus ((C \times D_{\uparrow})\times (D \times D_{\uparrow \uparrow})),$$
	where $D_{\uparrow \uparrow}=\{f \in D: \ f\mbox{ is strictly increasing}\}$. Taking into account the inclusion  $(U,L_\Phi) \in C \times D_{\uparrow}$ almost surely, we get  that
	\begin{equation*}
	\bP((U,L_\Phi)\in {\rm Disc}(g))=0.
	\end{equation*}
	Thus we can use the continuous mapping theorem (see \cite[Theorem $3.4.3$]{whitt2002stochastic}) to obtain the desired convergence.
\end{proof}
\section{Convergence properties of the generalized Fokker-Planck equation associated to the TCfOU}\label{Sec6}
Let us now consider only the case in which $a_\Phi=b_\Phi=0$ and $\nu_\Phi(0,+\infty)=+\infty$. It is established in  \cite{ascione2019time}   that $p_{H,\Phi}$ is solution (in some sense) of a generalized Fokker-Planck equation. To introduce such equation, we first need to define some suitable operators, following the results of  \cite{toaldo2015convolution} and \cite{ascione2020alpha}.
\begin{defn}\label{def1}
	The \textbf{Caputo-type non-local derivative} induced by $\Phi$ of an absolutely continuous function $u:[0,T]\to X$, where $X$ is a suitable Banach space, is defined as
	\begin{equation*}
		\partial^\Phi_t u(t)=\int_0^t \bar{\nu}_\Phi(t-\tau)u'(\tau)d\tau,
	\end{equation*}
	where $\bar{\nu}_\Phi(t)=\nu_\Phi(t,+\infty)$.\\
	We define the \textbf{$\Phi$-subordination} operator $\cS_\Phi: L^\infty(\R^+;X) \to L^\infty(\R^+;X)$ as
	\begin{equation*}
	\cS_\Phi v(t)=\int_0^{+\infty}v(s)f_\Phi(s;t)ds
	\end{equation*}
	and the \textbf{weighted $\Phi$-subordination} operator $\cS_{\Phi,H}:L^\infty(\R^+;X) \to L^\infty(\R^+;X)$ as
	\begin{equation*}
	\cS_{\Phi,H} v(t)=\int_0^{+\infty}V'_{2,H}(s)v(s)f_\Phi(s;t)ds.
	\end{equation*}
	Let us recall that, as proven in \cite{ascione2020alpha}, these two operators are continuous and injective.\\
	We define the \textbf{weighted Laplace transform} as
	\begin{equation*}
	L_H v(\lambda)=\cL_{t \to \lambda}[V'_{2,H}(t)v(t)]
	\end{equation*}
	where $\cL_{t \to \lambda}$ is the Laplace transform operator acting on $t$.\\
	Arguing as we did in \cite{ascione2020alpha}, let us recall that we can choose $c_1<0<c_2$ such that $c_1-c_2>-1/\theta$ and $\Phi^{-1}$ is defined on the vertical line $$r_{c_2}=\{\lambda \in \C: \ \lambda=c_2+iz, \ z \in \R\}$$ in such a way that $\Phi^{-1}(r_{c_2}) \subseteq \bH:=\{\lambda \in \C: \ \Re(\lambda)>0\}$. For any function $v:\bH \to \C$, let us define the operator $\widehat{L}_{H,\Phi}$ as
	\begin{multline}\label{defhatL}
	\widehat{L}_{H,\Phi}v(\lambda)=\frac{1}{4\pi^2}\int_0^{+\infty}e^{-\Phi(\lambda) t}\\\times \lim_{R \to +\infty}\int_{-\infty}^{+\infty}e^{(c_1+iw)t}\int_{-R}^{R}\cL_{t \to \lambda}[V'_{2,H}(t)](c_1-c_2+i(w-z))\\ \times\frac{\Phi^{-1}(c_2+iz)}{c_2+iz}v(\Phi^{-1}(c_2+iz))dzdwdt,
	\end{multline}
	whenever the involved integrals exist.\\
	Finally, we define the \textbf{generalized Fokker-Planck operator} $\cF_{H,\Phi}$ acting on $v \in L^\infty(\R^+; C^2(I))$ as
	\begin{equation*}
	\cF_{H,\Phi}v(t,x)=\cL^{-1}_{\lambda \to t}\left[\frac{\Phi(\lambda)}{\lambda}\pdsup{}{x}{2}\widehat{L}_{H,\Phi}\cL_{s \to \lambda}[v(s,x)]\right],
	\end{equation*}
	whenever all the involved operators are well-defined.
\end{defn}
Once we have defined the involved operators, we can consider the following \textbf{generalized Fokker-Planck equation} as
\begin{equation}\label{genFP}
\partial_t^\Phi v(t,x)=\frac{1}{2}\cF_{H,\Phi}v(t,x), \qquad x \in I, \ t>0,
\end{equation}
where $I \subseteq \R$.\\
In the case of subordinated functions, i.e. functions $v_\Phi=S_\Phi v$ for some $v \in L^\infty(\R^+;C^2(I))$, the operator $\widehat{L}_{H,\Phi}$ can be simplified. Indeed, we have the following result (see \cite[Proposition $4.5$]{ascione2020alpha}).
\begin{prop}\label{prop:simp}
Let $v_\Phi(t,x)=\cS_\Phi v(t,x)$ for some $v \in L^\infty(\R^+;C^2(I))$, where $I \subset \R$ is an interval and such that one of the following properties hold:
\begin{itemize}
	\item[$(a)$] $v$ is Lipschitz and $x \in \R \mapsto \cL[v](c_2+ix)$ belongs to $L^1(\R)$;
	\item[$(b)$] $v$ belongs to $L^2(\R^+)$ and $x \in \R \mapsto \cL[v](c_2+ix)$ belongs to $L^2(\R)$,
\end{itemize}
where $c_1,c_2$ are chosen in such a way that $c_1<0<c_2$ and $c_1-c_2>-1/\theta$ and $\Phi^{-1}$ is well-defined on the vertical line $r_{c_2}$. Then, for any $\lambda \in \bH$, it holds
\begin{equation*}
\widehat{L}_{H,\Phi}\cL_{t \to \lambda}[v_\Phi(t,x)](\lambda)=L_{H}v(\Phi(\lambda),x).
\end{equation*}
\end{prop}
With this result in mind, we have that for any subordinated function $v_\Phi=\cS_\Phi v$ satisfying the hypotheses of Proposition \ref{prop:simp} the generalized Fokker-Planck operator can be rewritten as
\begin{equation}\label{fpop}
\cF_{H,\Phi}v_\Phi(t,x)=\cL^{-1}_{\lambda \to t}\left[\frac{\Phi(\lambda)}{\lambda}\pdsup{}{x}{2}L_{H}v(\Phi(\lambda),x)\right].
\end{equation}
Now let us recall the following definitions of solutions. The general case is given in \cite{ascione2020alpha}; here we will consider only the case of subordinated solutions, thus the definitions will take in consideration the simplified formula \eqref{fpop} in place of the actual definition of $\cF_{H,\Phi}$.
\begin{defn}\label{def2}
	Let us consider $v_\Phi:I \times [0,+\infty) \to \R$ such that there exists $v \in L^\infty(\R^+;C^0(I))$ such that $v_\Phi=\cS_\Phi v$. We say that $v_\Phi$ is a \textbf{classical solution} of equation \eqref{genFP}
	if
	\begin{itemize}
		\item $v_\Phi$ belongs to the domain of $\cF_{\Phi,H}$;
		\item $\partial^\Phi_t v_\Phi(x,\cdot)$ is well-defined for any $x \in I$;
		\item equation \eqref{genFP} holds pointwise for almost any $t \in [0,T]$ and any $x \in I$.
	\end{itemize}
	Moreover, we say that a classical solution $v_\Phi$ is a \textbf{strong solution} if, for any $x \in I$, $v_\Phi(\cdot,x)\in C^1(0,+\infty)$ and there exists $\varepsilon>0$ such that $v_\Phi(\cdot,x)\in W^{1,1}(0,\varepsilon)$.\\
	We say that $v_\Phi$ is a \textbf{mild solution} of equation \eqref{genFP} if
	\begin{itemize}
		\item $v_\Phi(\cdot,x)$ is Laplace transformable for any $x \in I$ with Laplace transform $\bar{v}_\Phi$;
		\item For any $\lambda \in \bH=\{\lambda \in \C: \ \Re(\lambda)>0\}$ it holds $\bar{v}_\Phi(\lambda,\cdot) \in C(I)$, where $C(I)$ is the space of continuous functions in $I$;
		\item It holds
		\begin{equation}\label{ltgenFP}
		\Phi(\lambda)\bar{v}_\Phi(\lambda,x)-\frac{\Phi(\lambda)}{\lambda}v_\Phi(0,x)=\frac{\Phi(\lambda)}{2\lambda}L_H v(\Phi(\lambda),x), \ x \in I, \ \lambda \in \bH.
		\end{equation}
	\end{itemize}
\end{defn}
\begin{rmk}
	Let us observe that equation \eqref{ltgenFP} can be obtained from equation \eqref{genFP} by applying the Laplace transform on its both sides  and  recalling that
	\begin{equation*}
	\cL_{t \to \lambda}[\partial_t^\Phi v_\Phi(t,x)]=\Phi(\lambda)\bar{v}_\Phi(\lambda,x)-\frac{\Phi(\lambda)}{\lambda}v(0,x).
	\end{equation*}
	Moreover, let us recall that the same definitions can be applied to the case $H=1/2$. However, in this case we need to pay attention to the fact that the operator $\cF_{\Phi,H}$ is defined only taking in consideration the fact that the parent process is Gaussian, hence $\cF_{\Phi,1/2}$ does not coincide with twice the generator of an Ornstein-Uhlenbeck process unless it is applied to $p_{\Phi,H}(x,t)$, as one can observe from \cite{gajda2015time,ascione2020non}.
\end{rmk}
From now on let us denote by $v_\Phi(t,x;H)$ any solution (mild or classical, depending on the case) of equation \eqref{genFP}. We want to show that if we consider any sequence $H_n \to 1/2^+$ such that $v_\Phi(t,x;H_n)$ converges in some sense to a function $v_\Phi(t,x;1/2)$, then the latter is still a mild or classical solution of equation \eqref{genFP}. To do this, we will need some preliminary technical results. Let us first recall that
\begin{equation}\label{eq:pass8}
V'_{2,\frac{1}{2}}(t)=e^{-\frac{2t}{\theta}},
\end{equation}
while, according to formula (5.3) from \cite{ascione2019time}, as $H> \frac{1}{2}$,
\begin{equation}\label{Vprime}
V'_{2,H}(t)=2H(2H-1)e^{-\frac{2t}{\theta}}\int_0^{t}e^{\frac{z}{\theta}}z^{2H-2}dz.
\end{equation}
By using these formulas, we can show the following technical lemma.
\begin{lem}\label{convVprime}
	For any $\varepsilon>0$ there exists a constant $H_\varepsilon \in \left(\frac{1}{2},1\right)$ and a function $C_{\varepsilon}:\left(\frac{1}{2},H_\varepsilon\right]\to \R_+$ such that $\lim_{H \to \frac{1}{2}^+}C_\varepsilon(H)=0$ and
	\begin{equation*}
	|V'_{2,H}(t)-V'_{2,\frac{1}{2}}(t)|\le C_\varepsilon(H)e^{-\frac{t}{\theta}} \quad \forall t \in [\varepsilon,+\infty), \  \ \forall H \in \left(\frac{1}{2},H_\varepsilon\right].
	\end{equation*}
	Moreover, for any $\varepsilon>0$, it holds that
	\begin{equation*}
	\limsup_{H \to \frac{1}{2}^+}\Norm{V'_{2,H}-V'_{2,\frac{1}{2}}}{L^\infty(0,\varepsilon)}\le 1.
	\end{equation*}
\end{lem}
\begin{proof}
	First of all, we need to find a power series representation of the integral involved in equation \eqref{Vprime}.
	Applying \cite[formula $3.383.1$]{gradshteyn2014table}, we get the equality
	\begin{equation*}
		\int_0^{t}e^{\frac{z}{\theta}}z^{2H-2}dz=\sum_{k=0}^{+\infty}\frac{1}{2H-1+k}\frac{t^k}{\theta^kk!}
		\end{equation*}
	and then, substituting it in equation \eqref{Vprime}, we come to the conclusion
		\begin{align*}
		V'_{2,H}(t)&=2H t^{2H-1}e^{-\frac{2t}{\theta}}+2H(2H-1) t^{2H-1}e^{-\frac{2t}{\theta}}\sum_{k=1}^{+\infty}\frac{1}{2H-1+k}\frac{t^k}{\theta^kk!}.
		\end{align*}
	Therefore, recalling also \eqref{eq:pass8}, the difference of the derivatives can be bounded as
	\begin{align}\label{passVprime}
	\begin{split}
	|V'_{2,H}(t)-V'_{2,\frac{1}{2}}(t)|&\le e^{-\frac{2}{\theta}t}|2Ht^{2H-1}-1|\\
	 &\qquad+2H(2H-1)e^{-\frac{t}{\theta}}\sum_{k=1}^{+\infty}\frac{1}{2H-1+k}\frac{t^{2H-1+k}e^{-\frac{t}{\theta}}}{\theta^kk!}\\
	&=:I_1(t,H)+I_2(t,H)
	\end{split}
	\end{align}
	Let us first work with the term $I_2(t,H)$, that contains the series. Define the function $g(t)=t^{2H-1+k}e^{-\frac{t}{\theta}}>0$ and   observe that
	\begin{equation*}
	g'(t)=\frac{e^{-\frac{t}{\theta}}t^{2H-2+k}}{\theta}(\theta(2H-1+k)-t).
	\end{equation*}
	Thus, $g(t)$ admits a maximum point in $t=\theta(2H-1+k)$ and it means that
	\begin{equation*}
	g(t)\le \theta^{2H-1+k}(2H-1+k)^{2H-1+k}e^{-(2H-1+k)},
	\end{equation*}
	\textcolor{black}{that} immediately implies that
	\begin{equation}\label{eq:pass9}
	I_2(t,H)\le 2H(2H-1)\theta^{2H-1}e^{-\frac{t}{\theta}}\sum_{k=1}^{+\infty}(2H-1+k)^{2H-2+k}\frac{e^{-(2H-1+k)}}{k!}.
	\end{equation}
	Let us also recall that
	\begin{equation*}
	k!\ge \sqrt{2\pi}e^{-k}k^{k+\frac{1}{2}},
	\end{equation*}
	to conclude from equation \eqref{eq:pass9} that
	\begin{equation}\label{eq:pass10}
	I_2(t,H)\le \frac{2H(2H-1)\theta^{2H-1}}{e^{2H-1}}e^{-\frac{t}{\theta}}\sum_{k=1}^{+\infty}\left(\frac{2H-1+k}{k}\right)^{k}\frac{(2H-1+k)^{2H-2}}{(2\pi k)^\frac{1}{2}}.
	\end{equation}
	Furthermore,  observing that
	\begin{equation*}
	\left(\frac{2H-1+k}{k}\right)^{k}=\left(\left(1+\frac{2H-1}{k}\right)^{\frac{k}{2H-1}}\right)^{2H-1},
	\end{equation*}
	we have that $\lim_{k \to +\infty}\left(\frac{2H-1+k}{k}\right)^{k}=e^{2H-1}$. Thus, to make sure that the series in equation \eqref{eq:pass10} converges, we only need to show that
	\begin{equation*}
	2-2H+\frac{1}{2}>1,
	\end{equation*}
	that is equivalent to the upper bound  $H<\frac{3}{4}$. So, we understand that choosing $H_\varepsilon<\frac{3}{4}$ and
	\begin{equation*}
	 C_1(H)=\frac{2H(2H-1)\theta^{2H-1}}{e^{2H-1}}\sum_{k=1}^{+\infty}\left(\frac{2H-1+k}{k}\right)^{k}\frac{(2H-1+k)^{2H-2}}{(2\pi k)^\frac{1}{2}},
	\end{equation*}
	we obtain that
	\begin{equation}\label{eq:pass12}
	I_2(t,H)\le C_1(H)e^{-\frac{t}{\theta}}
	\end{equation}
	for any $H \in \left(\frac{1}{2},H_\varepsilon\right]$, and, moreover,  $C_1(H)\to 0$ as $H \to \frac{1}{2}^+$ by dominated convergence.\\
	Now let us consider $I_1(t,H)$ as defined in inequality \eqref{passVprime}, i.e.
	\begin{equation}\label{eq:pass10b}
	I_1(t,H)=|2Ht^{2H-1}-1|e^{-\frac{2}{\theta}t}.
	\end{equation}
	It is natural to distinguish three cases. If $t \in [0,\varepsilon]$, we can state that
	\begin{equation*}
	I_1(t,H)\le |2Ht^{2H-1}-1|
	\end{equation*}
	where $2Ht^{2H-1}-1$ is an increasing function. Hence,
	\begin{equation*}
	I_1(t,H)\le \max\{1,|2H\varepsilon^{2H-1}-1|\}.
	\end{equation*}
	In such a case, calculating the supremum in    $t \in [0,\varepsilon]$ in inequality \eqref{passVprime}, we obtain that
	\begin{equation*}
	\Norm{V'_{2,H}-V'_{2,\frac{1}{2}}}{L^\infty(0,\varepsilon)}\le \max\{1,|2H\varepsilon^{2H-1}-1|\}+C_1(H),
	\end{equation*}
	and   taking the limit superior as $H \to \frac{1}{2}^+$, we ultimately come to the upper bound
	\begin{equation*}
	\limsup_{H \to \frac{1}{2}^+}\Norm{V'_{2,H}-V'_{2,\frac{1}{2}}}{L^\infty(0,\varepsilon)}\le 1,
	\end{equation*}
	since $|2H\varepsilon^{2H-1}-1|\to 0$.\\
	Now let us consider $t \in [\varepsilon,1]$. For this values of argument we have the upper bound
	\begin{equation*}
	I_1(t,H)\le e^{-\frac{2t}{\theta}}\max\left\{|2H\varepsilon^{2H-1}-1|,2H-1\right\}:=C_2(H)e^{- \frac{2t}{\theta}}
	\end{equation*}
	where $C_2(H)\to 0$ as $H \to \frac{1}{2}^+$. Finally, let us consider $t \in (1,+\infty)$. Let us observe that $2Ht^{2H-1}-1>0$ if and only if $t^{2H-1}>\frac{1}{2H}$, where $2H>1$. In particular, this is achieved if $t>1$. Therefore, by equality \eqref{eq:pass10b}, we get that
	\begin{equation*}
	I_1(t,H)=e^{-\frac{t}{\theta}}f_H(t),
	\end{equation*}
	where
	\begin{equation}\label{eq:pass11}
	f_H(t)=(2Ht^{2H-1}-1)e^{-\frac{t}{\theta}}.
	\end{equation}
	Setting \textcolor{black}{$f_H(+\infty)=0$}, we can state that  $f_H$ is a continuous   and non-negative function on the interval $[1,+\infty]$. So, we can search for a maximum within this set. Differentiating $f_H$,  one can see that
	\begin{equation*}
	f'_H(t)=\frac{e^{-\frac{t}{\theta}}}{\theta}(2H\theta(2H-1)t^{2H-2}-2Ht^{2H-1}+1).
	\end{equation*}
	Denote by $t_{max}(H)$ the maximum point of $f_H$. Then, since $f_H(+\infty)=0$, it is possible to conclude that either
	$t_{max}(H)=1$, or
	\begin{equation*}
	2H(\theta(2H-1)t_{max}(H)^{-1}-1)t_{max}(H)^{2H-1}+1=0.
	\end{equation*}
	The latter equality is equivalent to the following one:
	\textcolor{black}{\begin{equation*}
		t_{max}(H)^{2H-1}=\frac{t_{max}(H)}{2H(t_{max}(H)-\theta(2H-1))}.
		\end{equation*}}
	If $t_{max}(H)=1$, then, evidently, by equation \eqref{eq:pass11}
	\begin{equation*}
	f_H(t_{max}(H))=(2H-1)e^{-\frac{1}{\theta}},
	\end{equation*}
	and this value  goes to $0$ as $H \to \frac{1}{2}$.
	If $t_{max}(H)\not = 1$, then, still by equation \eqref{eq:pass11},
	\begin{equation*}
	f_H(t_{max}(H))=\left(\frac{1}{1-\frac{\theta(2H-1)}{t_{max}(H)}}-1\right)e^{-\frac{t_{max}(H)}{\theta}}.
	\end{equation*}
	\textcolor{black}{Since $H>\frac{1}{2}$ and $t_{max}(H)\ge1$, we have
		\begin{equation*}
		\left|\frac{\theta(2H-1)}{t_{max}(H)}\right|\le \theta(2H-1)
		\end{equation*}
		and then we have that $\lim_{H \to \frac{1}{2}^+}\frac{\theta(2H-1)}{t_{max}(H)}=0.$ From this observation it is easy to conclude that $\lim_{H \to \frac{1}{2}^+}f_H(t_{max}(H))=0$.}
	Thus we can define $C_3(H)=f_H(t_{max}(H))$ and $C_4(H)=\max\{C_2(H),C_3(H)\}$. In such case, for any $t \in [\varepsilon,+\infty)$ and $H \in \left(\frac{1}{2},H_\varepsilon\right]$ we have, recalling equations \eqref{passVprime} and \eqref{eq:pass12}, that
	\begin{equation*}
	|V'_{2,H}(t)-V'_{2,\frac{1}{2}}(t)|\le (C_4(H)+C_1(H))e^{-\frac{t}{\theta}},
	\end{equation*}
	concluding the proof by setting $C_\varepsilon(H)=C_4(H)+C_1(H)$.
\end{proof}
\begin{rmk}
	Let us observe that the previous lemma implies the uniform convergence of $V'_{2,H}(t)$ towards $V'_{2,1/2}(t)$ in any interval of the form $(\varepsilon,+\infty)$ for $\varepsilon>0$.
\end{rmk}
Now we are ready to show the main result of this Section.
\begin{thm}\label{thm:mildconv}
	Let us consider a sequence $H_n \to \frac{1}{2}^+$ and a sequence of mild solutions $v_\Phi(t,x;H_n)$ of equation \eqref{genFP} for each $n \in \N$ such that $v_\Phi(t,x;H_n)=\cS_\Phi v(t,x;H_n)$ for $v(\cdot,\cdot;H_n) \in L^\infty((0,+\infty);C^2(I))$. Let us suppose there exists a function $v(\cdot,\cdot;1/2)\in L^\infty((0,+\infty);C^2(I))$ such that $v(\cdot,\cdot;H_n) \to v(\cdot,\cdot;1/2)$ strongly, i.e.
	\begin{equation*}
	\lim_{n \to +\infty}\sup_{t \in (0,+\infty)}\Norm{v(t,\cdot;H_n)-v(t,\cdot;1/2)}{C^2(I)}=0.
	\end{equation*}
	where with $\sup$ we intend the essential supremum and
	\begin{equation*}
	\Norm{f}{C^2(I)}=\sum_{i=0}^{2}\Norm{\dersup{}{x}{i}f(x)}{L^\infty(I)}.
	\end{equation*}
	Moreover, let us suppose that $v(\cdot,x;1/2)$ and $v(\cdot,x;H_n)$ satisfy the hypotheses of Proposition \ref{prop:simp} for any $n \in \N$ and $x \in I$. Then $v_\Phi(\cdot,\cdot;1/2)=S_\Phi v(\cdot,\cdot;1/2)$ is a mild solution of equation \eqref{genFP} for $H=1/2$.\\
Let us suppose additionally $v(\cdot,\cdot;H_n) \in \cD(\cF_{H_n,\Phi},I)$, where $\cD(\cF_{H_n,\Phi},I)$ is the domain of the operator $\cF_{H_n,\Phi}$, $\cF_{H_n,\Phi}v(t,\cdot;H_n) \in C^0(I)$ and, for fixed $x \in I$, $\cF_{H_n}v(\cdot,x;H_n)\in L^\infty(0,+\infty)$, where $\cF_{H_n}:=V'_{2,H_n}\pdsup{}{x}{2}$. Then $v_\Phi(\cdot,\cdot;H_n)$ are classical solutions of \eqref{genFP} and $v_\Phi(\cdot,\cdot;1/2)$ is a classical solution of \eqref{genFP} for $H=1/2$.
\end{thm}
\begin{proof}
	Let us first observe that, since $\cS_\Phi$ is a bounded linear operator it holds $v_\Phi(\cdot,\cdot;H_n) \to v_\Phi(\cdot,\cdot;1/2)$ strongly in $L^\infty(\R^+;C^{2}(I))$.\\
	Moreover, let us recall that, by definition of mild solution in Definition \ref{def2},
	\begin{equation}\label{eq:pass13a}
	 \Phi(\lambda)\overline{v}_\Phi(\lambda,x;H_n)-\frac{\Phi(\lambda)}{\lambda}v_{\Phi}(0,x;H_n)-\frac{\Phi(\lambda)}{2\lambda}\pdsup{}{x}{2}L_{H_n}v(\Phi(\lambda),x;H_n)=0
	\end{equation}
	for any $\lambda \in \bH$ and $x \in I$. Thus we have, subtracting the left-hand side of equation \eqref{eq:pass13a},
	\begin{align}\label{eq:pass13b}
	\begin{split}
	 &\Phi(\lambda)\overline{v}_\Phi(\lambda,x;1/2)-\frac{\Phi(\lambda)}{\lambda}v(0,x;1/2)-\frac{\Phi(\lambda)}{2\lambda}\pdsup{}{x}{2}L_{\frac{1}{2}}v(\Phi(\lambda),x;1/2)\\
	&\quad =\Phi(\lambda)\overline{v}_\Phi(\lambda,x;1/2)-\frac{\Phi(\lambda)}{\lambda}v(0,x;1/2)-\frac{\Phi(\lambda)}{2\lambda}\pdsup{}{x}{2}L_{\frac{1}{2}}v(\Phi(\lambda),x;1/2)\\
	&\qquad -\Phi(\lambda)\overline{v}_\Phi(\lambda,x;H_n)+\frac{\Phi(\lambda)}{\lambda}v_{\Phi}(0,x;H_n)+\frac{\Phi(\lambda)}{2\lambda}\pdsup{}{x}{2}L_{H_n}v(\Phi(\lambda),x;H_n).
	\end{split}
	\end{align}
	Now let us recall (see \cite[Proposition $4.2$]{ascione2020alpha})that
	\begin{equation*}
	\overline{v}_\Phi(\lambda,x;H)=\frac{\Phi(\lambda)}{\lambda}\overline{v}(\lambda,x;H)
	\end{equation*}
	and that $v_\Phi(0,x;H)=v(0,x;H)$ for any $x \in I$, $\lambda \in \bH$ and $H \in \left[\frac{1}{2},1\right)$. Hence, using these relations in equation \eqref{eq:pass13b}, we achieve
	\begin{align*}
	 &\Phi(\lambda)\overline{v}_\Phi(\lambda,x;1/2)-\frac{\Phi(\lambda)}{\lambda}v(x,0;1/2)-\frac{\Phi(\lambda)}{2\lambda}\pdsup{}{x}{2}L_{\frac{1}{2}}v(\Phi(\lambda),x;1/2)\\
	&\quad =\frac{\Phi^2(\lambda)}{\lambda}(\overline{v}(\Phi(\lambda),x;1/2)-\overline{v}(\Phi(\lambda),x;H_n))\\
	&\qquad +\frac{\Phi(\lambda)}{\lambda}(v(0,x;H_n)-v(0,x;1/2))\\
	&\qquad +\frac{\Phi(\lambda)}{2\lambda}\left(\pdsup{}{x}{2}L_{\frac{1}{2}}v(\Phi(\lambda),x;H_n)-\pdsup{}{x}{2}L_{\frac{1}{2}}v(\Phi(\lambda),x;1/2)\right)\\
	\end{align*}
	Using the triangular inequality, we obtain
	\begin{align*}
	 &\left|\Phi(\lambda)\overline{v}_\Phi(\lambda,x;1/2)-\frac{\Phi(\lambda)}{\lambda}v(0,x;1/2)-\frac{\Phi(\lambda)}{2\lambda}\pdsup{}{x}{2}L_{\frac{1}{2}}v(\Phi(\lambda),x;1/2)\right|\\
	 &\quad \le \frac{\Phi^2(\lambda)}{\lambda}\left|\overline{v}(\Phi(\lambda),x;H_n)-\overline{v}(\Phi(\lambda),x;1/2)\right|\\
	& \qquad + \frac{\Phi(\lambda)}{\lambda}\left|v(0,x;H_n)-v(0,x;1/2)\right|\\
	& \qquad + \frac{\Phi(\lambda)}{2\lambda}\left|\pdsup{}{x}{2}L_{H_n}v(\Phi(\lambda),x;H_n)-\pdsup{}{x}{2}L_{\frac{1}{2}}v(\Phi(\lambda),x;1/2)\right|
	\end{align*}
	Now let us control the second summand on the right hand side with the $L^\infty(\R^+;C^2(I))$ distance and let us add and subtract the term $\pdsup{}{x}{2}L_{1/2}v(\Phi(\lambda),x;H_n)$ in the third summand. With another application of the triangular inequality we get:
	\begin{align*}
	 &\left|\Phi(\lambda)\overline{v}_\Phi(\lambda,x;1/2)-\frac{\Phi(\lambda)}{\lambda}v(0,x;1/2)-\frac{\Phi(\lambda)}{2\lambda}\pdsup{}{x}{2}L_{\frac{1}{2}}v(\Phi(\lambda),x;1/2)\right|\\
	&\quad \le \frac{\Phi^2(\lambda)}{\lambda}\left|\overline{v}(\Phi(\lambda),x;H_n)-\overline{v}(\Phi(\lambda),x;1/2)\right|\\
	&\qquad +\frac{\Phi(\lambda)}{\lambda}\Norm{v(\cdot,\cdot;H_n)-v(\cdot,\cdot;1/2)}{L^\infty(\R^+;C^2(I))}\\
	& \qquad+ \frac{\Phi(\lambda)}{2\lambda}\left|\pdsup{}{x}{2}\left(L_{H_n}v(\Phi(\lambda),x;H_n)-L_{\frac{1}{2}}v(\Phi(\lambda),x;H_n)\right)\right|\\
	& \qquad+ \frac{\Phi(\lambda)}{2\lambda}\left|\pdsup{}{x}{2}\left(L_{\frac{1}{2}}v(\Phi(\lambda),x;H_n)-L_{\frac{1}{2}}v(\Phi(\lambda),x;1/2)\right)\right|.
	\end{align*}
	Now let us make another estimate. Let us observe that
	\begin{align*}
	|\overline{v}(\Phi(\lambda),x;H_n)-\overline{v}(\Phi(\lambda),x;1/2)|&\le \int_0^{+\infty}e^{-\Phi(\lambda)t}|v(t,x;H_n)-v(t,x;1/2)|dt\\
	& \le \frac{1}{\Phi(\lambda)}\Norm{v(\cdot,\cdot;H_n)-v(\cdot,\cdot;1/2)}{L^\infty(\R^+;C^2(I))},
	\end{align*}
	for any $x \in I$ and $\lambda \in \bH$. Thus we finally obtain
	\begin{align}\label{eq:pass17}
	\begin{split}
	 &\left|\Phi(\lambda)\overline{v}_\Phi(\lambda,x;1/2)-\frac{\Phi(\lambda)}{\lambda}v(0,x;1/2)-\frac{\Phi(\lambda)}{2\lambda}\pdsup{}{x}{2}L_{\frac{1}{2}}v(\Phi(\lambda),x;1/2)\right|\\
	&\quad \le  \frac{\Phi(\lambda)}{\lambda}\left(2\Norm{v(\cdot,\cdot;H_n)-v(\cdot,\cdot;1/2)}{L^\infty(\R^+;C^2(I))}+\frac{1}{2}(I_1(x,\lambda)+I_2(x,\lambda))\right),
	\end{split}
	\end{align}
	where
	\begin{align}\label{eq:pass13}
	\begin{split}
	 I_1(x,\lambda)&:=\left|\pdsup{}{x}{2}\left(L_{\frac{1}{2}}v(\Phi(\lambda),x;H_n)-L_{\frac{1}{2}}v(\Phi(\lambda),x;H_n)\right)\right|,\\
	 I_2(x,\lambda)&:=\left|\pdsup{}{x}{2}\left(L_{\frac{1}{2}}v(\Phi(\lambda),x;H_n)-L_{\frac{1}{2}}v(\Phi(\lambda),x;1/2)\right)\right|.
	\end{split}
	\end{align}
	Let us first work with $I_1(x,\lambda)$. We have, since the Laplace transform is a linear operator and $\pdsup{}{x}{2}$ is a closed operator (see \cite{arendt}), recalling the definition of $L_H$, given in Definition \ref{def1},
	\begin{align}\label{eq:pass14}
	\begin{split}
	I_1(x,\lambda)&=\pdsup{}{x}{2}\left(\cL_{t \to \lambda}\left[(V'_{2,H_n}(t)-V'_{2,1/2}(t))v(t,x;H_n)\right]\right)(\Phi(\lambda))\\
	&=\cL_{t \to \lambda}\left[(V'_{2,H_n}(t)-V'_{2,1/2}(t))\pdsup{}{x}{2}v(t,x;H_n)\right](\Phi(\lambda))\\
	&=\int_0^{+\infty}e^{-\Phi(\lambda) t}(V_{2,H_n}'(t)-V_{2,\frac{1}{2}}'(t))\pdsup{}{x}{2}v(t,x;H_n)dt.
	\end{split}
	\end{align}
	Fix $\varepsilon>0$ and define $H_\varepsilon \in \left(\frac{1}{2},1\right)$ as in Lemma \ref{convVprime}. Supposing, without loss of generality, that $H_n<H_\varepsilon$ for any $n \in \N$, we have, by Lemma \ref{convVprime},
	\begin{align*}
	I_1(x,\lambda)&=\int_0^{\varepsilon}e^{-\Phi(\lambda) t}(V_{2,H_n}'(t)-V_{2,\frac{1}{2}}'(t))\pdsup{}{x}{2}v(t,x;H_n)dt\\
	&+\int_{\varepsilon}^{+\infty}e^{-\Phi(\lambda) t}(V_{2,H_n}'(t)-V_{2,\frac{1}{2}}'(t))\pdsup{}{x}{2}v(t,x;H_n)dt\\
	&\le \Norm{V'_{2,H_n}-V'_{2,\frac{1}{2}}}{L^\infty(0,\varepsilon)}\Norm{v_n}{L^\infty(\R^+;C^2(I))}\frac{1-e^{-\Phi(\lambda) \varepsilon}}{\Phi(\lambda)}\\
	&\quad +\frac{C_\varepsilon(H_n)}{\Phi(\lambda)+\frac{1}{\theta}}e^{-\left(\Phi(\lambda)+\frac{1}{\theta}\right)\varepsilon}\Norm{v_n}{L^\infty(\R^+;C^2(I))}\\
	&\le \left(\Norm{V'_{2,H_n}-V'_{2,\frac{1}{2}}}{L^\infty(0,\varepsilon)}\frac{1-e^{-\Phi(\lambda) \varepsilon}}{\Phi(\lambda)}+\theta C_\varepsilon(H_n)\right)\Norm{v_n}{L^\infty(\R^+;C^2(I))}.
	\end{align*}
	Since $v_n \to v$ strongly in $L^\infty(\R^+;C^2(I))$, there exists a constant $K$ such that \linebreak $\Norm{v_n}{L^\infty(\R^+;C^2(I))}\le K$ for any $n \in \N$. Thus
	\begin{equation}\label{eq:pass16}
	I_1(x,\lambda)\le K\left(\Norm{V'_{2,H_n}-V'_{2,\frac{1}{2}}}{L^\infty(0,\varepsilon)}\frac{1-e^{-\Phi(\lambda) \varepsilon}}{\Phi(\lambda)}+\theta C_\varepsilon(H_n)\right).
	\end{equation}
	Concerning $I_2(x,\lambda)$ defined in equation \eqref{eq:pass13}, we have, by the definition of $L_H$ and equation \eqref{eq:pass8}, arguing as we did for $I_1(x,\lambda)$ in equation \eqref{eq:pass14},
	\begin{equation*}
	I_2(x,\lambda)=\int_0^{+\infty}e^{-\left(\Phi(\lambda)+\frac{2}{\theta}\right)t}\pdsup{}{x}{2}(v(t,x;H_n)-v(t,x;1/2))dt.
	\end{equation*}
	and then
	\begin{align}\label{eq:pass15}
	\begin{split}
	I_2(x,\lambda)&\le \frac{1}{\Phi(\lambda)+\frac{2}{\theta}}\Norm{v(\cdot,\cdot;H_n)-v(\cdot,\cdot;H)}{L^\infty(\R^+;C^2(I))}\\
	&\le \frac{\theta}{2}\Norm{v(\cdot,\cdot;H_n)-v(\cdot,\cdot;1/2)}{L^\infty(\R^+;C^2(I))}.
	\end{split}
	\end{align}
	We conclude that, using equations \eqref{eq:pass16} and \eqref{eq:pass15} in equation \eqref{eq:pass17},
	\begin{align*}
	 &\left|\Phi(\lambda)\overline{v}_\Phi(x,\lambda;1/2)-\frac{\Phi(\lambda)}{\lambda}v(0,x;1/2)-\frac{\Phi(\lambda)}{2\lambda}\pdsup{}{x}{2}L_{\frac{1}{2}}v(\Phi(\lambda),x;1/2)\right|\\
	& \qquad \le  \frac{\Phi(\lambda)}{\lambda}\left(\left(2+\frac{\theta}{4}\right)\Norm{v(\cdot,\cdot;H_n)-v(\cdot,\cdot;1/2)}{L^\infty(\R^+;C^2(I))}\right.\\
	&\qquad\qquad\left. +\frac{K}{2}\left(\Norm{V'_{2,H_n}-V'_{2,\frac{1}{2}}}{L^\infty(0,\varepsilon)}\frac{1-e^{-\Phi(\lambda) \varepsilon}}{\Phi(\lambda)}+\theta C_\varepsilon(H_n)\right)\right),
	\end{align*}
	Thus we have, taking the supremum over $I$, for fixed $\lambda \in \cH$,
	\begin{align*}
	 &\Norm{\Phi(\lambda)\overline{v}_\Phi(\lambda,\cdot;1/2)-\frac{\Phi(\lambda)}{\lambda}v(0,\cdot;1/2)-\frac{\Phi(\lambda)}{2\lambda}\pdsup{}{x}{2}L_{\frac{1}{2}}v(\Phi(\lambda),\cdot;1/2)}{L^\infty(I)}\\
	& \qquad \le  \frac{\Phi(\lambda)}{\lambda}\left(\left(2+\frac{\theta}{4}\right)\Norm{v(\cdot,\cdot;H_n)-v(\cdot,\cdot;1/2)}{L^\infty(\R^+;C^2(I))}\right.\\
	&\qquad\qquad\left. +\frac{K}{2}\left(\Norm{V'_{2,H_n}-V'_{2,\frac{1}{2}}}{L^\infty(0,\varepsilon)}\frac{1-e^{-\Phi(\lambda) \varepsilon}}{\Phi(\lambda)}+\theta C_\varepsilon(H_n)\right)\right),
	\end{align*}
	Taking the limit superior as $n \to +\infty$, we obtain, by Proposition \ref{convVprime} and the fact that $v(\cdot,\cdot;H_n) \to v(\cdot,\cdot;1/2)$ strongly in $L^\infty(\R^+;C^2(I))$,
	\begin{align*}
	 &\Norm{\Phi(\lambda)\overline{v}_\Phi(\cdot,\lambda;1/2)-\frac{\Phi(\lambda)}{\lambda}v(\cdot,0;1/2)-\frac{\Phi(\lambda)}{2\lambda}\pdsup{}{x}{2}L_{\frac{1}{2}}v(\cdot,\Phi(\lambda);1/2)}{L^\infty(I)}\\
	& \qquad \le  \frac{K(1-e^{-\Phi(\lambda) \varepsilon})}{2\lambda}.
	\end{align*}
	Now we can send $\varepsilon \to 0$ to obtain that, pointwise, for $x \in I$ and $\Re(\lambda)>0$, it holds
	\begin{equation*}
	 \Phi(\lambda)\overline{v}_\Phi(\lambda,x;1/2)-\frac{\Phi(\lambda)}{\lambda}v(0,x;1/2)-\frac{\Phi(\lambda)}{2\lambda}\pdsup{}{x}{2}L_{\frac{1}{2}}v(\Phi(\lambda),x;1/2)=0
	\end{equation*}

Now, under the additional hypotheses, we know that $v_{\Phi}(\cdot,\cdot;H_n)$ are classical solutions by \cite[Theorem $5.5$]{ascione2020alpha}, which is a gain-of-regularity theorem. We know also that $v \in L^\infty(\R^+;C^2(I))$. Now let us observe that
	\begin{equation*}
	\cF_{1/2}v(t,x;1/2)=e^{-\frac{2t}{\theta}}\pdsup{}{x}{2}v(t,x;1/2)
	\end{equation*}
	is well defined as $v \in L^\infty(\R^+;C^2(I))$ and belongs to $L^\infty(\R^+)$ for fixed $x \in I$, as $0 \le e^{-\frac{2t}{\theta}}\le 1$. Hence, since we have shown that $v$ is a mild solution, by the same gain-of-regularity theorem as before, we have that $v$ is a classical solution, concluding the proof.	
\end{proof}
\begin{rmk}
	Let us observe that if we fix the initial data and the boundary data, by the weak maximum principle proved in \cite{ascione2020alpha}, the strong solutions are unique, thus, if we suppose that $v(\cdot,\cdot;H_n)$ are strong solutions, the convergence we obtain in the previous theorem is towards the unique strong solution of equation \eqref{genFP}.\\
	Let us also remark that Theorem \ref{thm:conv} provides a useful example for Theorem \ref{thm:mildconv}, as we have that, if $H_n \to 1/2^+$, then $p_{H_n, \Phi} \to p_{1/2,\Phi}$, where $p_{H_n, \Phi}$ are classical solutions of equation \eqref{genFP} and $p_{1/2,\Phi}$ is a classical solution of \eqref{genFP} with $H=1/2$. Another interesting case is given by $\Phi(\lambda)=\lambda^\alpha$. Indeed, in such case, it has been shown in \cite{ascione2020alpha} that $p_{H_n, \Phi}$ are strong solutions of equation \eqref{genFP} for $x \in \R\setminus \{0\}$ and $t>0$, and so it is $p_{1/2,\Phi}$. Thus, for fixed boundary data $$\lim_{|x| \to +\infty}p_{H,\Phi}(t,x)=0, \ t>0$$ and $$p_{H,\Phi}(t,0)=\frac{1}{\sqrt{2\pi}}\int_0^{+\infty}(V_{2,H}(s))^{-\frac{1}{2}}f_\Phi(s;t)ds, \ t>0$$
	and fixed initial datum $p_{H,\Phi}(0,x)=0$ for $x \in \R\setminus \{0\}$, we have that the unique strong solutions of equation \eqref{genFP} converge towards the unique strong solution of equation \eqref{genFP} for $H=1/2$.
\end{rmk}

\section*{Acknowledgements}
This research is partially supported by MIUR - PRIN 2017, project Stochastic Models for Complex Systems, no. 2017JFFHSH, by Gruppo Nazionale per il Calcolo Scientifico (GNCS-INdAM), by Gruppo Nazionale per l'Analisi Matematica, la Probabilit\`a e le loro Applicazioni (GNAMPA-INdAM). The research of the 2nd author was partially supported by ToppForsk project nr. 274410 of the Research Council of Norway with title STORM: Stochastics for Time-Space Risk Models.

\bibliographystyle{plain}
\bibliography{biblio}      

\begin{thebibliography}{10}

\bibitem{anh2005financial}
Vo~Anh and Akihiko Inoue.
\newblock Financial markets with memory {I}: {D}ynamic models.
\newblock {\em Stochastic Analysis and Applications}, 23(2):275--300, 2005.

\bibitem{arendt}
Wolfgang Arendt, Charles J~K Batty, Matthias Hieber, and Frank Neubrander.
\newblock {\em Vector-valued {L}aplace Transforms and {C}auchy Problems}.
\newblock Mono. Math. Birkh\"auser, Basel, 2001.

\bibitem{ascione2020non}
Giacomo Ascione, Nikolai Leonenko, and Enrica Pirozzi.
\newblock Non-local {P}earson diffusions.
\newblock {\em arXiv preprint arXiv:2009.12086}, 2020.

\bibitem{ascione2019fractional}
Giacomo Ascione, Yuliya Mishura, and Enrica Pirozzi.
\newblock Fractional {O}rnstein-{U}hlenbeck process with stochastic forcing,
  and its applications.
\newblock {\em Methodology and Computing in Applied Probability}, pages 1--32,
  2019.

\bibitem{ascione2020alpha}
Giacomo Ascione, Yuliya Mishura, and Enrica Pirozzi.
\newblock The {F}okker-{P}lanck equation for the time-changed fractional
  {O}rnstein-{U}hlenbeck process.
\newblock {\em arXiv preprint arXiv:2005.12628}, 2020.

\bibitem{ascione2019time}
Giacomo Ascione, Yuliya Mishura, and Enrica Pirozzi.
\newblock Time-changed fractional {O}rnstein-{U}hlenbeck process.
\newblock {\em Fractional Calculus and Applied Analysis}, 23(2):450--483, 2020.

\bibitem{ascione2019semi}
Giacomo Ascione and Bruno Toaldo.
\newblock A semi-{M}arkov {L}eaky {I}ntegrate-and-{F}ire model.
\newblock {\em Mathematics}, 7(11):1022, 2019.

\bibitem{bertoin1996levy}
Jean Bertoin.
\newblock {\em L{\'e}vy Processes}, volume 121.
\newblock Cambridge {U}niversity {P}ress, 1996.

\bibitem{cheridito2003fractional}
Patrick Cheridito, Hideyuki Kawaguchi, and Makoto Maejima.
\newblock Fractional {O}rnstein-{U}hlenbeck processes.
\newblock {\em Electronic Journal of Probability}, 8:3--14, 2003.

\bibitem{gajda2015time}
Janusz Gajda and Agnieszka Wy{\l}oma{\'n}ska.
\newblock Time-changed {O}rnstein--{U}hlenbeck process.
\newblock {\em Journal of Physics A: Mathematical and Theoretical},
  48(13):135004, 2015.

\bibitem{gradshteyn2014table}
Izrail~Solomonovich Gradshteyn and Iosif~Moiseevich Ryzhik.
\newblock {\em Table of Integrals, Series, and Products}.
\newblock Academic Press, 2014.

\bibitem{kaarakka2011fractional}
Terhi Kaarakka and Paavo Salminen.
\newblock On fractional {O}rnstein-{U}hlenbeck processes.
\newblock {\em Communications on Stochastic Analysis}, 5(1):8, 2011.

\bibitem{lamperti1962convergence}
John Lamperti.
\newblock On convergence of stochastic processes.
\newblock {\em Transactions of the American Mathematical Society},
  104(3):430--435, 1962.

\bibitem{leonenko2013correlation}
Nikolai~N Leonenko, Mark~M Meerschaert, and Alla Sikorskii.
\newblock Correlation structure of fractional {P}earson diffusions.
\newblock {\em Computers \& Mathematics with Applications}, 66(5):737--745,
  2013.

\bibitem{leonenko2013fractional}
Nikolai~N Leonenko, Mark~M Meerschaert, and Alla Sikorskii.
\newblock Fractional {P}earson diffusions.
\newblock {\em Journal of Mathematical Analysis and Applications},
  403(2):532--546, 2013.

\bibitem{meerschaert2008triangular}
Mark~M Meerschaert and Hans-Peter Scheffler.
\newblock Triangular array limits for continuous time random walks.
\newblock {\em Stochastic Processes and Their Applications}, 118(9):1606--1633,
  2008.

\bibitem{meerschaert2013inverse}
Mark~M Meerschaert and Peter Straka.
\newblock Inverse stable subordinators.
\newblock {\em Mathematical Modelling of Natural Phenomena}, 8(2):1--16, 2013.

\bibitem{mishura2018stochastic}
Yuliya Mishura, Vladimir Piterbarg, Kostiantyn Ralchenko, and Anton
  Yurchenko-Tytarenko.
\newblock Stochastic representation and path properties of a fractional
  {C}ox--{I}ngersoll--{R}oss process.
\newblock {\em Theory of Probability and Mathematical Statistics}, 97:167--182,
  2018.

\bibitem{mishura2018fractional}
Yuliya Mishura and Anton Yurchenko-Tytarenko.
\newblock Fractional {C}ox--{I}ngersoll--{R}oss process with non-zero mean.
\newblock {\em Modern Stochastics: Theory and Applications}, 5(1):99--111,
  2018.

\bibitem{sakai1999temporally}
Yutaka Sakai, Shintaro Funahashi, and Shigeru Shinomoto.
\newblock Temporally correlated inputs to leaky integrate-and-fire models can
  reproduce spiking statistics of cortical neurons.
\newblock {\em Neural Networks}, 12(7-8):1181--1190, 1999.

\bibitem{schilling2012bernstein}
Ren{\'e}~L Schilling, Renming Song, and Zoran Vondracek.
\newblock {\em Bernstein Functions: Theory and Applications}, volume~37.
\newblock Walter de Gruyter, 2012.

\bibitem{shinomoto1999ornstein}
Shigeru Shinomoto, Yutaka Sakai, and Shintaro Funahashi.
\newblock The {O}rnstein-{U}hlenbeck process does not reproduce spiking
  statistics of neurons in prefrontal cortex.
\newblock {\em Neural Computation}, 11(4):935--951, 1999.

\bibitem{silvestrov2012limit}
Dmitrii~S Silvestrov.
\newblock {\em Limit Theorems for Randomly Stopped Stochastic Processes}.
\newblock Springer Science \& Business Media, 2012.

\bibitem{toaldo2015convolution}
Bruno Toaldo.
\newblock Convolution-type derivatives, hitting-times of subordinators and
  time-changed ${C}_0$-semigroups.
\newblock {\em Potential Analysis}, 42(1):115--140, 2015.

\bibitem{totoki1962method}
Haruo Totoki.
\newblock A method of construction of measures on function spaces and its
  applications to stochastic processes.
\newblock {\em Memoirs of the Faculty of Science, Kyushu University. Series A,
  Mathematics}, 15(2):178--190, 1962.

\bibitem{whitt2002stochastic}
Ward Whitt.
\newblock {\em Stochastic-Process Limits: an Introduction to Stochastic-Process
  Limits and Their Application to Queues}.
\newblock Springer Science \& Business Media, 2002.

\end{thebibliography}
\end{document}